\documentclass{amsart}

\usepackage[utf8]{inputenc}
\usepackage{amssymb, amscd}
\usepackage[all]{xy}
\usepackage{graphicx}
\usepackage{hyperref}

\numberwithin{equation}{section}

\newtheorem{thm}{Theorem}[section]

\newtheorem{lem}[thm]{Lemma}
\newtheorem{cor}[thm]{Corollary}

\newtheorem{q}[thm]{Question}

\theoremstyle{definition}

\theoremstyle{remark}
\newtheorem{rem}[thm]{Remark}

\renewcommand{\hom}{\operatorname{Hom}}
\renewcommand{\ker}{\operatorname{Ker}}

\newcommand{\Z}{\mathbb{Z}}

\newcommand{\R}{\mathbb{R}}
\newcommand{\C}{\mathbb{C}}

\newcommand{\mediumoplus}[2]{\mbox{\footnotesize$\displaystyle\bigoplus\limits_{#1}^{#2}$}}

\DeclareMathOperator{\aut}{Aut}

\DeclareMathOperator{\rank}{rank}

\DeclareMathOperator{\tr}{tr}

\def\op{\operatorname}

\begin{document}

\title[Representation varieties detect essential surfaces]
{Representation varieties detect essential surfaces}
\author[S.~Friedl]{Stefan Friedl}
\address{Department of Mathematics, University of Regensburg, Germany}
\email{sfriedl@gmail.com}
\author[T.~Kitayama]{Takahiro KITAYAMA}
\address{Graduate School of Mathematical Sciences, the University of Tokyo, Japan}
\email{kitayama@ms.u-tokyo.ac.jp}
\author[M.~Nagel]{Matthias Nagel}
\address{Universit\'e du Qu\'ebec \`a Montr\'eal, Qu\'ebec, Canada}
\email{nagel@cirget.ca}
\subjclass[2010]{Primary~57N10, Secondary~57M05, 20E42}
\keywords{3-manifold, essential surface, character variety, Bruhat-Tits building, separable subgroup}

\begin{abstract}
Extending Culler-Shalen theory, Hara and the second author presented a way to
construct certain kinds of branched surfaces in a $3$-manifold from an ideal
point of a curve in the $\op{SL}_n$-character variety.  There exists an
essential surface in some $3$-manifold known to be not detected in the
classical $\op{SL}_2$-theory.  We prove that every connected essential surface
in a $3$-manifold is given by an ideal point of a rational curve in the
$\op{SL}_n$-character variety for some $n$.
\end{abstract}

\maketitle

\section{Introduction}

In this paper we study an extension of Culler-Shalen theory for
higher-dimensional representations.  In their seminal work~\cite{CS} Culler and
Shalen established a method to construct essential surfaces in a $3$-manifold
from an ideal point of a curve in the $\op{SL}_2(\C)$-character variety. The
method is built on a beautiful combination of the theory of incompressible
surfaces in a $3$-manifold, the geometry of representation varieties, and
Bass-Serre theory~\cite{Se1, Se2}. We refer the reader to the
exposition~\cite{Sh} for literature and related topics on Culler-Shalen theory.
Hara and the second author presented an analogous extension of the
Culler-Shalen method to the case of higher-dimensional representations~\cite{HK}.
They showed that certain kinds of branched surfaces (possibly without any branched points) are constructed from an ideal point of a curve in the $\op{SL}_n(\C)$-character variety for a general $n$.
Such a branched surface corresponds to a nontrivial splitting of the $3$-manifold group as a complex of groups~\cite{C, Ha}.

The classical theory for $2$-dimensional representations is not sufficient to
detect all essential surfaces in Haken manifolds.  Throughout the paper let $M$
be a compact connected orientable $3$-manifold.  We denote by $X_n(M)$ the
$\op{SL}_n(\C)$-character variety of $\pi_1 M$.  It was discovered by Boyer and
Zhang~\cite{BZ}, and Motegi~\cite{Mo} that there exist infinitely many Haken
manifolds $M$, which are even hyperbolic, such that $X_2(M)$ has no irreducible
component of positive dimension.  See also \cite{SZ} for further study on the
topic.  We say that \textit{an essential surface $S$ in $M$ is given by an
ideal point $\chi$ of a curve in $X_n(M)$} if $S$ is constructed from $\chi$ by
the Culler-Shalen method or its extension developed in \cite{HK} as described
in Subsection~\ref{subsec_HK}.  Hara and the second author formulated and
raised the following question~\cite[Question 6.1]{HK}.

\begin{q} \label{q_HK}
Does there exist an essential surface in some $3$-manifold $M$ not given by any ideal point of curves in $X_2(M)$ but given by an ideal point of a curve in $X_n(M)$ for some $n$?
\end{q}

The aim of this paper is to show that the extension of Culler-Shalen theory to the case of higher-dimensional representations~\cite{HK} detects all essential surfaces in Haken manifolds. 
The following is the main theorem of this paper, which, in particular, gives an affirmative answer to Question~\ref{q_HK}.

\begin{thm} \label{thm_main}
Every connected essential surface in $M$ is given by an ideal point of a rational curve in $X_n(M)$ for some $n$.
\end{thm}

The proof of Theorem \ref{thm_main} relies on the breakthroughs of Agol~\cite{A} and Wise~\cite{W}, and the subsequent works of Przytycki and Wise~\cite{PW1, PW2} on the separability of subgroups in a $3$-manifold group.
For a given connected essential surface $S$ in $M$ there exists a non-separating lift $T$ of $S$ in some finite cover $N$ of $M$ by the separability of $\pi_1 S$ in $\pi_1 M$.
The non-separating surface $T$ defines abelian representations $\pi_1 N \to \op{SL}_2(\C)$ parameterized in $\C^\times$, which induces an affine curve $D_T$ consisting of representations $\pi_1 M \to \op{SL}_n(\C)$ where $n$ is twice the degree of the cover $N$.
The set of characters of representations in $D_T$ is a desired rational curve in $X_n(M)$ as in the statement of Theorem \ref{thm_main}, which has a unique ideal point.
Then analyzing the structure of the Bruhat-Tits building associated to the function field of $D_T$,
we explicitly construct a PL-map from the universal cover of $M$ to the $1$-skeleton of the building. Finally, we show that the inverse image of midpoints of edges by the PL-map is isotopic to parallel copies of $S$.

The paper is organized as follows. In Section~\ref{sec_HK} we review the
extension of the Culler-Shalen method to the case of
$\op{SL}_n(\C)$-representations in \cite{HK}.  Here we recall some of the
standard facts on $\op{SL}_n(\C)$-character varieties and  Bruhat-Tits
buildings associated to the special linear group. In Subsection~\ref{subsec_HK}, 
we give the precise definition of the sentence
`\textit{an essential surface is given by an ideal point}'.  Section
\ref{sec_PW} provides a brief exposition on the separability of surface
subgroups by Przytycki and Wise~\cite{PW2}.  Section \ref{sec_proof} is devoted
to the proof of Theorem \ref{thm_main}.

\subsection*{Acknowledgment}
We wish to thank Steven Boyer and Takashi Hara for helpful conversations.
The first and the third author was supported by the SFB 1085 `Higher Invariants' at the Universit\"at Regensburg funded by the Deutsche Forschungsgemeinschaft (DFG). 
The second author was supported by JSPS KAKENHI (No.\ 26800032).

\section{$\op{SL}_n$-Culler Shalen theory}\label{sec_HK}

We begin with an overview of the extension of Culler-Shalen theory to the case of higher-dimensional representations in \cite{HK}.

\subsection{Character varieties}

We briefly review the $\op{SL}_n(\C)$-character variety of a finitely generated group.
See \cite{LM, Si1, Si2} for more details.

Let $\pi$ be a finitely generated group.
We define the following affine algebraic set 
\[ R_n(\pi) = \hom(\pi, \op{SL}_n(\C)). \]
The algebraic group $\op{SL}_n(\C)$ acts on the affine algebraic set $R_n(\pi)$ by conjugation.
We denote by $X_n(\pi)$ the GIT quotient of the action~\cite{MFK}:
\[ X_n(\pi) = \hom(\pi, \op{SL}_n(\C)) // \op{SL}_n(\C). \]
The affine algebraic set $X_n(\pi)$ is called the \textit{$\op{SL}_n(\C)$-character variety} of $\pi$.
By definition the coordinate ring $\C[X_n(\pi)]$ is isomorphic to the subring $\C[R_n(\pi)]^{\op{SL}_n(\C)}$ of $\C[R_n(\pi)]$ consisting of $\op{SL}_n(\C)$-invariant functions.
Procesi~\cite[Theorem 1.3]{P} showed that $\C[R_n(\C)]^{\op{SL}_n(\C)}$ is generated by \textit{trace functions} $I_\gamma$ for $\gamma \in \pi$ defined by 
\[ I_\gamma(\rho) = \tr \rho(\gamma) \]
for $\rho \in R_n(\pi)$.
Therefore $X_n(\C)$ is identified with the set of \textit{characters} $\chi_\rho$ for $\rho \in R_n(\pi)$ defined by
\[ \chi_\rho(\gamma) = \tr \rho(\gamma) \]
for $\gamma \in \pi$.
For a compact connected orientable $3$-manifold $M$, we abbreviate $R_n(\pi_1 M)$ and $X_n(\pi_1 M)$ with $R_n(M)$ and $X_n(M)$ respectively to simplify notation.

Let $C$ be an affine variety, and denote by $\C(C)$ its field of rational functions.
We call $C$ an \textit{affine curve} if the transcendence degree of $\C(C)$ over $\C$ equals $1$~\cite[Section 6.5]{F}.
Consider an affine curve $C$ and its projectivisation $\overline C$.
The projective curve~$\overline C$ might not be smooth, but it has a unique smooth model, i.e., there is a smooth projective curve $\widetilde C$ together with a birational map $\widetilde C \dashrightarrow \overline C$ which is universal~\cite[Theorem 7.3]{F}.
Recall that a birational equivalence induces an isomorphism on the associated fields of rational functions \cite[Proposition 6.12]{F}.
Thus their fields of rational functions all agree:
$\C(C) = \C(\overline C) = \C(\widetilde C)$.
To a point $P$ of $\widetilde C$ the local ring~$\mathcal{O}_{\widetilde C, P}$ of $\widetilde C$ at $P$ is associated.
As the point $P$ is a smooth point, the ring $\mathcal{O}_{\widetilde C, P}$ is a discrete
valuation ring, which induces a discrete valuation $v_P$ on $\C(C)$ \cite[Section 7.1]{F}.

An \textit{ideal point} $\chi$ of an affine curve $C$ is a point of its smooth projective model $\widetilde C$ corresponding to a point of $\overline C \setminus C$.
We can equip the rational functions $\C(C)$ with the discrete valuation~$v_\chi$ 
associated to an ideal point~$\chi$ as described above.

\subsection{Bruhat-Tits buildings}

Following the exposition \cite{G}, we describe the Bruhat-Tits building~\cite{BT1, BT2, IM} associated to the special linear group over a discrete valuation field.
See also \cite{AB} for more details on buildings.

Let $F$ be a commutative field equipped with a discrete valuation $v$ which is not necessarily complete.
We denote by $\mathcal{O}_v$ the valuation ring associated to $v$.
The \textit{Bruhat-Tits building} associated to $\op{SL}_n(F)$, which is an $(n-1)$-dimensional simplicial complex $B_v$, is defined as follows:
A vertex of $B_v$ is the homothety class of a lattice in the $n$-dimensional vector space $F^n$, where a \textit{lattice} in $F^n$ is a free $\mathcal{O}_v$-submodule of full rank, and two lattices $\Lambda$ and $\Lambda'$ are \textit{homothetic} if $\Lambda = \alpha \Lambda'$ for some $\alpha \in F^\times$.
A set of $(m+1)$ vertices $s_0, s_1, \dots, s_m$ forms an $m$-simplex in $B_v$ if and only if there exist lattices $\Lambda_0, \Lambda_1, \dots, \Lambda_m$ representing $s_0, s_1, \dots, s_m$ respectively such that after relabeling indices we have the flag relation
\[ \omega \Lambda_m \subsetneq \Lambda_0 \subsetneq \Lambda_1 \subsetneq \dots \subsetneq \Lambda_m, \]
where $\omega$ is an irreducible element of $\mathcal{O}_v$.

The simplicial complex $B_v$ is known to be an \textit{Euclidean building}, and, in particular, a $\op{CAT}(0)$-space with respect to the standard metric.
See for instance \cite[Definition 11.1]{AB} for the definition of an Euclidean building.
Since $\op{SL}_n(F)$ acts on the set of lattices in $F^n$ so that homothety classes and above flag relations are preserved, $\op{SL}_n(F)$ acts also on $B_v$.
This action is \textit{type-preserving}, i.e., there exists an $\op{SL}_n(F)$-invariant map $\tau \colon B_v^{(0)} \to \Z / n \Z$ such that $\tau|_{\Delta^{(0)}}$ is a bijection for each $(n-1)$-simplex $\Delta$ in $B_v$.
Here for a simplicial complex $K$ we denote by $K^{(m)}$ the $m$-skeleton of $K$.
In particular, for any subgroup $G$ of $\op{SL}_n(F)$ the quotient $B_v / G$ is again an $(n-1)$-dimensional simplicial complex.

\begin{rem}
In the case of $n = 2$ the above construction is nothing but the one of the tree associated to $\op{SL}_2(F)$ in \cite{Se1, Se2}.
\end{rem}

\subsection{An ideal point giving an essential surface} \label{subsec_HK}

We summarize the construction in~\cite{HK} of a certain branched surface from an ideal point of a curve in the character variety.
Here we restrict our attention to the case where such a branched surface has no branched points, and is an essential surface.

Let $C$ be a curve in $X_n(M)$ and $\chi$ an ideal point of $C$.
We denote by $t \colon R_n(M) \to X_n(M)$ the quotient map.
There exists a curve $D$ in $t^{-1}(C)$ such that $t|_D$ is not a constant map, and a regular map $\tilde t|_D \colon \widetilde D \to \widetilde C$ on the smooth projective models is induced by $t|_D$.
We take a lift $\tilde \chi \in (\tilde t|_D)^{-1}(\{ \chi \})$, and denote by $B_{\tilde \chi}$ the Bruhat-Tits building associated to $\op{SL}_n(\C(D))$, where the function field $\C(D)$ is equipped with the discrete valuation at $\tilde \chi$.
The \textit{tautological representation} $\mathcal{P} \colon \pi_1 M \to \op{SL}_n(\C(D))$ is defined by
\[ \mathcal{P}(\gamma)(\rho) = \rho(\gamma) \]
for $\gamma \in \pi_1 M$ and $\rho \in D$.
Pulling back the action of $\op{SL}_n(\C(D))$ on $B_{\tilde \chi}$ by $\mathcal{P}$, we obtain the action of $\pi_1 M$ on $B_{\tilde \chi}$.
Extending \cite[Theorem 2.2.1]{CS} to the case of a general $n$, Hara and the second author \cite[Corollary 4.5]{HK} proved that the action is \textit{nontrivial}, i.e., for every vertex of $B_{\tilde \chi}$ its stabilizer subgroup of $\pi_1 M$ is proper.

Recall that a compact orientable properly-embedded surface $S$ in $M$ is called \textit{essential} if for any component $S_0$ of $S$ the inclusion-induced homomorphism $\pi_1 S_0 \to \pi_1 N$ is injective, and $S_0$ is not boundary-parallel nor homeomorphic to the $2$-sphere $S^2$.
We say that \textit{an essential surface $S$ is given by an ideal point $\chi$} if for some lift $\tilde \chi$ of $\chi$ there exists a PL map $f \colon M \to B_{\tilde \chi}^{(1)} / \pi_1 M$ whose inverse image of the set of midpoints of the edges in $B_{\tilde \chi}$ is isotopic to some number of parallel copies of $S$.

When $n = 2$, since $\pi_1 M$ nontrivially acts on the tree $B_{\tilde \chi}$ without inversions, every ideal point $\chi$ gives some essential surface in $M$~\cite[Proposition 2.3.1]{CS}.
In general, it follows from the proof of \cite[Theorem 4.7]{HK} that if $n = 3$ or if $\partial M$ is non-empty, then there exists a PL map $f \colon M \to B_{\tilde \chi}^{(2)} / \pi_1 M$ such that $f^{-1}(Y)$ is a certain branched surface called \textit{essential tribranched surface}~\cite[Definition 2.2]{HK}, where $Y$ is the union of edges in the first barycentric subdivision of $B_{\tilde \chi}^{(2)} / \pi_1 M$ not contained in $B_{\tilde \chi}^{(1)} / \pi_1 M$.
Note that an essential tribranched surface without any branched points is nothing but an essential surface in the usual sense.

\begin{rem}
The authors~\cite{FKN} showed that every closed $3$-manifold $M$ with $\rank \pi_1 M \geq 4$ contains an essential tribranched surface.
\end{rem}

\section{Surface subgroup separability} \label{sec_PW}

We recall the separability of surface subgroups in a $3$-manifold group proved by Przytycki and Wise~\cite{PW2}, which is a key ingredient of the proof of Theorem~\ref{thm_main}.
A subgroup $H$ of a group $G$ is \textit{separable} if $H$ equals the intersection of finite index subgroups of $G$ containing $H$.

\begin{thm}$(${\cite[Theorem 1.1]{PW2}}$)$ \label{thm_PW}
Let $S$ be a connected essential surface $S$ in $M$.
Then $\pi_1 S$ is separable in $\pi_1 M$.
\end{thm}

Theorem \ref{thm_PW} was proved by Przytycki and Wise~\cite{PW1} when $M$ is a graph manifold, and by Wise~\cite{W} when $M$ is a hyperbolic manifold.
In fact, every finitely generated subgroup of $\pi_1 M$ is separable when $M$ is a hyperbolic manifold, by Wise~\cite{W} in the case where $M$ has a non-empty boundary and by Agol~\cite{A} in the case where $M$ is closed.
See also Liu~\cite{L} for a refinement of the separability.

The following is a topological interpretation of Theorem~\ref{thm_PW}.
While it is well-known for experts, nevertheless we give a proof for the sake of completeness.
See also \cite[Lemma 1.4]{Sc}.

\begin{cor} \label{cor_PW}
For an essential surface $S$ in $M$ there exists some finite cover of $M$ where the inverse image of $S$ contains a non-separating component.
\end{cor}

\begin{proof}
We may assume that $S$ is connected and separating.
Let $M_+$ and $M_-$ be the two components of the complement of $S$.
It follows from \cite[Theorem~10.5]{He} that $\pi_1 S$ has index at least two in $\pi_1 M_-$ and in $\pi_1 M_+$.
It follows from Theorem~\ref{thm_PW} that there exists an epimorphism $\varphi \colon \pi_1 M \to G$ to a finite group such that $\varphi(\pi_1 S) \neq \varphi(\pi_1 M_\pm)$.
In particular, we have
\[ [G ~\colon~ \varphi(\pi_1 S)] \geq 2 [G ~\colon~ \varphi(\pi_1 M_\pm)]. \]
Let $p \colon M_\varphi \to M$ be the covering corresponding to $\ker \varphi$.
The numbers of components of $p^{-1}(M_\pm)$ and $p^{-1}(S)$ are equal to $[G \colon \varphi(\pi_1 M_\pm)]$ and $[G ~\colon~ \varphi(\pi_1 S)]$ respectively, and the above inequality implies
\[[G ~\colon~ \varphi(\pi_1 S)] \geq [G ~\colon~ \varphi(\pi_1 M_+)] + [G ~\colon~ \varphi(\pi_1 M_-)]. \]
Thus the number of components of $p^{-1}(S)$ is greater than or equal to that of its complement, which shows that some component of $p^{-1}(S)$ is non-separating.
\end{proof}

\section{Proof of the main theorem} \label{sec_proof}

Now we prove the main theorem.
For the readers' convenience we recall the statement.

\begin{thm}[Theorem \ref{thm_main}]
Every connected essential surface in $M$ is given by an the ideal point of a rational curve in $X_n(M)$ for some $n$.
\end{thm}

Let $S$ be a connected essential surface in $M$.
It follows from Corollary~\ref{cor_PW} that there exists a $d$-fold covering $p \colon N \to M$ for some $d$ such that $p^{-1}(S)$ contains a non-separating component $T$.
Then the proof is divided into two parts:
First, we construct a rational curve $C_T$ in $X_{2d}(M)$ with a unique ideal point $\chi_T$, which is determined by $T$.
Second, for a lift $\tilde \chi_T$ of $\chi_T$ we construct a PL map $f \colon M \to B_{\tilde \chi_T}^{(1)}$ whose inverse image of the set of midpoints of edges in $B_{\tilde \chi_T}$ is isotopic to two parallel copies of $S$.

\subsection{Construction of a curve}

We denote by $\psi \colon \pi_1 N \to \Z$ the epimorphism induced by the intersection pairing with $T$.
For each $z \in \C^\times$ we define the representation $\tilde \rho_z \colon \pi_1 N \to \op{SL}_2(\C)$ to be the composition of $\psi$ and the homomorphism $\Z \to \op{SL}_2(\C)$ which sends an integer $k$ to the matrix
\[
\begin{pmatrix}
z^k & 0 \\
0 & z^{-k}
\end{pmatrix}
. \]
We consider the induced representation $\rho_z \colon \pi_1 M \to \aut(\C[\pi_1 M] \otimes_{\C[\pi_1 N]} \C^2)$ of $\tilde\rho_z$.
Fixing representatives $\gamma_1, \dots, \gamma_d \in \pi_1 M$ of the elements of $\pi_1 M / p_*(\pi_1 N)$, we have the decomposition
\[ \C[\pi_1 M] \otimes_{\C[\pi_1 N]} \C^2\, =\, \mediumoplus{i = 1}{d} \gamma_i \otimes \C^2, \]
which is naturally identified with $\C^{2d}$.
Thus we regard $\rho_z$ as a representation $\pi_1 M \to \op{SL}_{2d}(\C)$.
We now set
\begin{align*}
D_T &= \{ \rho_z \in R_{2d}(M) ~:~ z \in \C^\times \}, \\
C_T &= \{ \chi_{\rho_z} \in X_{2d}(M) ~:~ z \in \C^\times \}.
\end{align*}

\begin{lem}
\begin{enumerate}
\item The set $D_T$ is a curve in $R_{2d}(M)$ isomorphic to $\C^\times$.
\item The set $C_T$ is a rational curve in $X_{2d}(M)$ with a unique ideal point.
\end{enumerate}
\end{lem}

\begin{proof}
These sets of representations and characters are constructed along the following commutative diagram: 
\[
\begin{CD}
\C^\times @>>> R_2(\Z) @>>> R_2(N) @>>> R_{2d}(M) \\
@VVV @VVV @VVV @VVV \\ 
\C @>>> X_2(\Z) @>>> X_2(N) @>>> X_{2d}(M),
\end{CD}
\]
where the first vertical map sends $z \in \C^\times$ to $z + z^{-1} \in \C$, and the first bottom horizontal map is an isomorphism which sends $w \in \C$ to the character of $\Z$ whose image of $1 \in \Z$ is $w$.
The composition of the top horizontal maps is called $\Psi$ and the composition of the bottom horizontal maps is called $\Phi$.
Then the sets $D_T$ and $C_T$ coincide with the images of $\Psi$ and $\Phi$

We may assume $\gamma_1 \in \pi_1 N$, and take $\mu \in \pi_1 N$ with $\psi(\mu) = 1$. 
Then we have
\begin{align*}
\rho_z(\mu) &= 
\begin{pmatrix}
z & 0 \\
0 & z^{-1}
\end{pmatrix}
\oplus \mediumoplus{i = 2}{d}
\begin{pmatrix}
z^{\psi(\gamma_i^{-1} \mu \gamma_i)} & 0 \\
0 & z^{-\psi(\gamma_i^{-1} \mu \gamma_i)}
\end{pmatrix}, \\
\chi_{\rho_z}(\mu) &= z + z^{-1} + \sum_{i = 2}^d \left( z^{\psi(\gamma_i^{-1} \mu \gamma_i)} + z^{-\psi(\gamma_i^{-1} \mu \gamma_i)} \right).
\end{align*}
Hence the restriction of the map $R_{2d}(M) \to \C^2$ sending a representation $\rho$ to the vector of the $(1, 1)$- and $(2, 2)$-entries of $\rho(\mu)$ gives the inverse regular map $D_T \to \C^\times$ of $\Psi$, where $\C^\times$ is identified with the curve $xy - 1$ in $\C^2$, and $(1)$ is proved.

We deduce from the second equation above that the map $\Phi$ is not constant.
Also by fixing an affine space $\C^N$ containing $X_{2d}(M)$, we regard $\Phi$ as a map $\C \to \C^N$. 
We denote by $\overline \Phi \colon \mathbb{P}^1 \to \mathbb{P}^N$ the projective extension of $\Phi$.
Since $\overline \Phi$ is not a constant map, by the completeness of the projective line $\mathbb{P}^1$~\cite[Section I.9, Theorem 1]{Mu} the image $\overline C_T$ of $\overline \Phi$ is a projective curve, and by the Riemann-Hurwitz formula the curve $\overline C_T$ is rational.
Therefore the set $C_T$, which coincides with the intersection of $\overline C_T$ and $\C^N$, is an affine rational curve.
Since $\Phi$ induces a surjective regular map $\widetilde \Phi \colon \mathbb{P}^1 \to \widetilde C_T$ on the smooth projective models, the rational curve $C_T$ has a unique ideal point corresponding to the point at infinity of $\mathbb{P}^1$, which proves $(2)$.
\end{proof}

It is a simple matter to check that both the two ideal points of $D_T$ corresponding to $0$ and $\infty$ are lifts of the unique ideal point $\chi_T$ of $C_T$.
Let $\tilde \chi_T$ be the one corresponding to $0$.
Then as in Section \ref{subsec_HK} we obtain the nontrivial action $\pi_1 M$ on the Bruhat-Tits building $B_{\tilde \chi_T}$ associated to $\op{SL}_{2d}(\C(D_T))$.
We identify $\C(D_T)$ with the standard function field $\C(t)$ and the valuation at $\tilde \chi_T$ with the lowest degree of the Laurent expansion of a rational function.
Then the vector space $\C(t)^{2d}$ is decomposed into
\[ \C[\pi_1 M] \otimes_{\C[\pi_1 N]} \C(t)^2 \,=\, \mediumoplus{i = 1}{d} \gamma_i \otimes \C(t)^2, \]
where $\pi_1 N$ acts on $\C(t)^2$ by the representation $\mathcal{Q} \colon \pi_1 N \to \op{SL}_2(\C(t))$ defined by
\[ \mathcal{Q}(\gamma) =
\begin{pmatrix}
t^{\psi(\gamma)} & 0 \\
0 & t^{-\psi(\gamma)}
\end{pmatrix}
\] 
for $\gamma \in \pi_1 N$, and the tautological representation $\mathcal{P} \colon \pi_1 M \to \op{SL}_{2d}(\C(t))$ is given by the left multiplication on $\C[\pi_1 M] \otimes_{\C[\pi_1 N]} \C(t)^2$.

\subsection{Construction of a PL-map}

We take a triangulation of $M$ containing $S$ as a normal surface.
We may assume that the intersection of each tetrahedron with $S$ is connected, if necessary, replacing the triangulation by an appropriate subdivision.
The triangulation of $M$ induces ones of $N$ and the universal cover $\widetilde M$ of $M$, so that $T$ and its inverse image $\widetilde T$ by the covering $\widetilde M \to N$ are also normal surfaces. Then we take a cellular map $g \colon N \to \R / \Z$ such that $g^{-1}([\frac{1}{2}]) = T$, where we consider the cellular structure of $\R / \Z$ with one vertex corresponding to $\Z$.   
We define $\tilde g \colon \widetilde M \to \R$ to be the $\pi_1 N$-equivariant lift of $g$, so that $\tilde g^{-1}(\frac{1}{2} + \Z) = \widetilde T$.  

We now define a map $\tilde{f}^{(0)} \colon \widetilde M^{(0)} \to B_{\tilde \chi_T}^{(0)}$ as follows.
For $s \in \widetilde M^{(0)}$ we consider the lattice
\[ \mediumoplus{i=1}{d} \gamma_i \otimes \Lambda_{\tilde g(\gamma_i^{-1} s)} \]
in $\C[\pi_1 M] \otimes_{\C[\pi_1 N]} \C(t)^2$, where $\Lambda_n$ is the lattice in $\C(t)^2$ generated by the vectors
\[
\begin{pmatrix}
t^n \\
0
\end{pmatrix}
~\text{and}~
\begin{pmatrix}
0 \\
t^{-n}
\end{pmatrix}
. \]
Note that $\tilde g(\gamma_i^{-1} s)$ is an integer for each $i$ by the construction of $\tilde g$.
Then we set $\tilde{f}^{(0)}(s)$ to be the homothety class of the above lattice.
In the following two lemmas we observe the key properties of $\tilde f^{(0)}$.

\begin{lem} \label{lem_A}
The map $\tilde f^{(0)}$ is $\pi_1 M$-equivariant.
\end{lem}

\begin{proof}
For $\gamma \in \pi_1 M$ there exist a permutation $\sigma$ of degree $d$ and $\delta_i \in \pi_1 N$ such that 
\[ \gamma \gamma_i = \gamma_{\sigma(i)} \delta_i \]
for each $i$.
Then 
\[ \begin{array}{rclclcl}
\bigoplus\limits_{i=1}^d \gamma \gamma_i \otimes \Lambda_{\tilde{g}(\gamma_i^{-1} s)} &= &\bigoplus\limits_{i=1}^d \gamma_{\sigma(i)} \delta_i \otimes \Lambda_{\tilde{g}(\gamma_i^{-1} s)} &=& \bigoplus\limits_{i=1}^d \gamma_{\sigma(i)} \otimes \mathcal{Q}(\delta_i) \cdot \Lambda_{\tilde{g}(\gamma_i^{-1} s)} \\
&=& \bigoplus\limits_{i=1}^d \gamma_{\sigma(i)} \otimes \Lambda_{\tilde{g}(\gamma_i^{-1} s) + \psi(\delta_i)} 
&= &\bigoplus\limits_{i=1}^d \gamma_{\sigma(i)} \otimes \Lambda_{\tilde{g}(\delta_i \gamma_i^{-1} s)} \\
&=& \bigoplus\limits_{i=1}^d \gamma_{\sigma(i)} \otimes \Lambda_{\tilde{g}(\gamma_{\sigma(i)}^{-1} \gamma s)} 
&= &\bigoplus\limits_{i=1}^d \gamma_i \otimes \Lambda_{\tilde{g}(\gamma_i^{-1} \gamma s)}\end{array}
\]
for $\gamma \in \pi_1 M$ and $s \in \widetilde M^{(0)}$, which implies that $\tilde f^{(0)}$ is a $\pi_1 M$-equivariant map.
\end{proof}

\begin{lem} \label{lem_B}
For each tetrahedron $\Delta$ in $\widetilde M$ the set $\tilde f^{(0)}(\Delta^{(0)})$ consists of one vertex if $\gamma_i^{-1} \cdot \Delta$ does not intersect with $\widetilde T$ for any $i$, and two vertices of distance $2$ with respect to the graph metric on $B_{\tilde \chi}^{(1)}$ otherwise.
\end{lem}

\begin{proof}
If $\gamma_i^{-1} \cdot \Delta$ does not intersect with $\widetilde T$ for any $i$, then it follows from the choice of $g$ that there exists some $n_i \in \Z$ such that
\[ \tilde g(\gamma_i^{-1} \cdot \Delta^{(0)}) = \{ n_i \} \]
for each $i$, and hence we obtain
\[ \tilde f^{(0)}(\gamma_i^{-1} \cdot \Delta^{(0)}) = \left\{ \left[  \mediumoplus{i = 1}{d} \gamma_i \otimes \Lambda_{n_i} \right] \right\}. \]

In the following we consider the case where $\gamma_i^{-1} \cdot \Delta$ intersects with $\widetilde T$ for $i = i_1, \dots, i_m$.  
Then the intersections of $\gamma_{i_k}^{-1} \cdot \Delta$ with $\widetilde T$ are all connected and of same type for $k = 1, \dots, m$, since otherwise $p(T) = S$ implies that the intersection of some tetrahedron in $M$ with $S$ is not a normal disc, which contradicts the choice of the triangulation of $M$.
Thus $\Delta^{(0)}$ is divided into two subsets $\Delta_+^{(0)}$ and $\Delta_-^{(0)}$ satisfying the following:
\begin{enumerate}
\item there exists some $n_{i_k} \in \Z$ such that
\[ \tilde g(\gamma_{i_k}^{-1} \cdot \Delta_+^{(0)}) = \{ n_{i_k} + 1 \} ~\text{and}~ \tilde g(\gamma_{i_k}^{-1} \cdot \Delta_-^{(0)}) = \{ n_{i_k} \} \]
for $k = 1, \dots, m$;
\item there exists some $n_i \in \Z$ such that
\[ \tilde g(\gamma_i^{-1} \cdot \Delta^{(0)}) = \{ n_i \} \]
for $i \neq i_1, \dots, i_m$.
\end{enumerate}
Hence we obtain
\[ \tilde f^{(0)}(\Delta_+^{(0)}) = \{ [\Lambda_+] \} ~\text{and}~ \tilde f^{(0)}(\Delta_-^{(0)}) = \{ [\Lambda_-] \}, \]
where
\begin{align*}
\Lambda_+ &= \bigg( \mediumoplus{k = 1}{m} \gamma_{i_k} \otimes \Lambda_{n_{i_k} + 1} \bigg) \oplus \bigg( \mediumoplus{i \neq i_1, \dots, i_m}{} \gamma_i \otimes \Lambda_{n_i} \bigg), \\
\Lambda_- &= \bigg( \mediumoplus{k = 1}{m} \gamma_{i_k} \otimes \Lambda_{n_{i_k}} \bigg) \oplus \bigg( \mediumoplus{i \neq i_1, \dots, i_m}{} \gamma_i \otimes \Lambda_{n_i} \bigg).
\end{align*}
Since
\begin{align*}
t \Lambda_n' \subsetneq \Lambda_{n+1} \subsetneq \Lambda_n', \\
t \Lambda_n' \subsetneq \Lambda_n \subsetneq \Lambda_n' 
\end{align*}
for $n \in \Z$, where $\Lambda_n'$ is the lattice in $\C(t)^2$ generated by the vectors
\[
\begin{pmatrix}
t^n \\
0
\end{pmatrix}
~\text{and}~
\begin{pmatrix}
0 \\
t^{-n-1}
\end{pmatrix}
, \]
we have
\begin{align*}
t \Lambda' \subsetneq \Lambda_+ \subsetneq \Lambda', \\
t \Lambda' \subsetneq \Lambda_- \subsetneq \Lambda',
\end{align*}
where
\[ \Lambda' = \bigg( \mediumoplus{k = 1}{m} \gamma_{i_k} \otimes \Lambda_{n_{i_k}}' \bigg) \oplus \bigg( \mediumoplus{i \neq i_1, \dots, i_m}{} \gamma_i \otimes \Lambda_{n_i} \bigg). \]
By the definition of the building $B_{\tilde{\chi}}$ these relations imply that there exist edges in $B_{\tilde{\chi}}$ connecting $[\Lambda_+]$ and $[\Lambda_-]$ with $[\Lambda']$, and hence the distance between $[\Lambda_+]$ and $[\Lambda_-]$ is at most $2$ in $B_{\tilde{\chi}}^{(1)}$.
We further observe that the matrix
\[ \left( \bigoplus_{i=1}^n \gamma_{i_k} \otimes
\begin{pmatrix}
t & 0 \\
0 & t^{-1}
\end{pmatrix}
\right) \oplus \left( \bigoplus_{i \neq i_1, \dots, i_m} \gamma_i \otimes
\begin{pmatrix}
1 & 0 \\
0 & 1
\end{pmatrix}
\right) \]
in $\op{SL}_{2d}(\C(t))$ sends $[\Lambda_-]$ to $[\Lambda_+]$.
Since the action of $\op{SL}_{2d}(\C(t))$ on $B_{\tilde{\chi}}$ is type-preserving, the distance between them is exactly equal to $2$, and the lemma follows.
\end{proof}

We are now in position to construct a desired PL map $f \colon M \to B_{\tilde \chi_T}^{(1)} / \pi_1 M$.
It follows from Lemmas \ref{lem_A} and \ref{lem_B} that $\tilde f^{(0)}$ extends to a $\pi_1 M$-equivariant simplicial map $\tilde f \colon \widetilde{N} \to B_{\tilde \chi_T}^{(1)}$ with respect to the $1$st barycentric subdivision of the triangulation of $\widetilde{M}$.
We define $f \colon M \to B_{\tilde \chi_T}^{(1)} / \pi_1 M$ to be the quotient of $\tilde f$ by $\pi_1 M$.
By the construction of $\tilde f$ we check at once that the inverse image of the set of midpoints of edges in $B_{\tilde \chi_T}$ by $\tilde f$ is isotopic to two parallel copies of $\widetilde T$.
Since $\tilde f$ is $\pi_1 M$-equivariant and since $p(T) = S$, the inverse image of the set of midpoints of edges in $B_{\tilde \chi_T} / \pi_1 M$ by $f$ is isotopic to two parallel copies of $S$.
Therefore $S$ is given by $\chi_T$, which completes the proof.


\end{document}